\documentclass[a4paper,10pt]{amsart}
\usepackage{amssymb,amsmath,amsfonts,amsthm}
\usepackage[dvips]{graphicx}
\usepackage{psfrag}
\usepackage{todonotes}

\theoremstyle{plain}
\newtheorem{main}{Theorem}

\newtheorem{maincor}[main]{Corollary}
\newtheorem{theorem}{Theorem}[section]
\newtheorem{lemma}[theorem]{Lemma}
\newtheorem{proposition}[theorem]{Proposition}
\newtheorem{corollary}[theorem]{Corollary}
\theoremstyle{remark}
\newtheorem{remark}[theorem]{Remark}
\newtheorem{definition}[theorem]{Definition}

\newcommand{\PH}{\operatorname{PH}}

           \def\ea{\end{array}}
          \def\ec{\end{center}}
     \def\ed{\end{description}}
        \def\ee{\end{equation}}
       \def\eea{\end{eqnarray}}
     \def\eeaa{\end{eqnarray*}}
 \def\et{\end{thebibliography}}

\def\Orb{{\rm Orb}}
\def\Diff{{\rm Diff}}

\def\HC{{\rm HC}}

\def\cU{{\mathcal U}}
\def\cV{{\mathcal V}}

\def\cF{{\mathcal F}}
\def\cM{{\mathcal M}}

\def\TT{{\mathbb T}}

\title{Shub's example revisited}
\author{Chao Liang, Radu Saghin, Fan Yang and Jiagang Yang}
\date{\today}
\thanks{Chao Liang has been supported by NNSFC grants   12271538   and 12071328 and the Disciplinary Funds in CUFE\\
Jiagang Yang was partially supported by CNPq, FAPERJ, PRONEX and NNSFC grants 11871487 and 12071202}

\address{School of Statistics and Mathematics, Central University of Finance and Economics,
Beijing, 100081, China}
\email{chaol\@@cufe.edu.cn}

\address{Instituto de Matem\'atica, Pontificia Universidad Cat\'olica de Valpara\'iso,
Blanco Viel 596, Cerro Bar\'on, Valpara\'iso-Chile.}
\email{radu.saghin\@@pucv.cl}

\address{Department of Mathematics, Michigan State University, MI, USA.}
\email{yangfa31\@@msu.edu}

\address{Departamento de Geometria, Instituto de Matem\'atica e Estat\'\i stica, Universidade Federal Fluminense, Niter\'oi, Brazil}
\email{yangjg\@@impa.br}

\begin{document}

\begin{abstract}
For a class of robustly transitive diffeomorphisms on $\TT^4$ introduced by Shub in \cite{Sh}, satisfying an additional bunching condition, we show that there exits a $C^2$ open and $C^r$ dense subset $\cU^r$, $2\leq r\leq\infty$, such that any two hyperbolic points of $g\in \cU^r$ with stable index $2$ are homoclinically related. As a consequence, every $g\in \cU^r$ admits a unique homoclinic class associated to the hyperbolic periodic points with index $2$, and this homoclinic class coincides to the whole ambient manifold. Moreover, every $g\in \cU^r$ admits at most one measure with maximal entropy, and every $g\in\cU^{\infty}$ admits a unique measure of maximal entropy.
\end{abstract}

\maketitle

\tableofcontents

\setcounter{tocdepth}{1} \tableofcontents

\section{Introduction and results}

Shub introduced in \cite{Sh} one of the first examples of robustly transitive diffeomorphisms (on $\mathbb T^4$) which are not uniformly hyperbolic. Although later Ma\~n\'e \cite{M} built a robustly transitive but non-hyperbolic example on $\TT^3$, the Shub's example still has its special interest, in particular because the center direction is two-dimensional and may admit no further domination.

In this paper we will consider a slightly more general class than the original setting of the Shub's example, a precise definition is the following.

\subsection{Shub class\label{ss.Shub example}}

\begin{definition} A diffeomorphism $f: M\to M$ is called \emph{partially hyperbolic} if the tangent bundle admits a continuous $Df$-invariant splitting $TM = E^s \oplus E^c\oplus E^u$ such that there exist continuous functions $0<\lambda_s(x)<\lambda_c^-(x)\leq\lambda_c^+(x)<\lambda_u(x)$, with $\lambda_s(x)<1<\lambda_u(x)$, satisfying the following conditions:
\begin{itemize}
\item[(1)] $\|Df(x)v^s\|\leq\lambda_s(x)$,
\item[(2)] $\lambda_c^-(x) \leq \|Df(x) v^c\| \leq \lambda_c^+(x)$,
\item[(3)] $\| Df(x) v^u\|\geq\lambda_u(x)$,
\end{itemize}
for every $x\in M$ and unit vectors $v^i \in E^i(x)$($i = s, c, u$).
\end{definition}

A partially hyperbolic diffeomorphism is called \emph{dynamically coherent} if there exist invariant foliations $\cF^{cs}$ and $\cF^{cu}$ tangent to $E^{cs} = E^c \oplus E^s$ and $E^{cu} = E^c \oplus E^u$. In this case $\cF^{cs}$ is subfoliated by the stable and central foliations $\cF^s$ and $\cF^c$, while $\cF^{cu}$ is subfoliated by the unstable and center foliations $\cF^u$ and $\cF^c$.

Let $A, B$ be two linear Anosov automormorphism on $\TT^2$ such that $1<|\lambda_B|<|\lambda_A|$, where $\lambda_A$ and $\lambda_B$ are the unstable eigenvalues of $A$ and $B$. Then $f_{A,B}(x,y): \TT^2\times \TT^2 \to \TT^2\times \TT^2 $
$$f_{A,B}(x,y)=(A(x),B(y))$$
is an Anosov automorphism, which can also be seen as a partially hyperbolic diffeomorphism with two-dimensional center bundle and one-dimensional stable and unstable bundles.

\begin{definition}
Let $\PH_{A,B}$ be the set of partially hyperbolic diffeomorphisms isotopic to $f_{A,B}$, all of them having the same dimension (i.e., one dimension) of the stable and unstable bundle, and let $\PH_{A,B}^0$ be the connected component of $\PH_{A,B}$ containing $f_{A,B}$.
\end{definition}

It is easy to see that $\PH_{A,B}^0$ is an open set of diffeomorphisms of $\TT^4$. It is known that:

\begin{proposition}\label{p.center classification}[Fisher-Potrie-Sambarino \cite{FPS}].
If $f\in \PH_{A,B}^0$, then $f$ is dynamically coherent and admits a center foliation where all central leaves are $C^1$ two-dimensional tori, and $f$ is center leaf conjugate to $f_{A,B}$.
\end{proposition}

\begin{definition}
The \emph{Shub class} $\mathcal{SH}\subset \cup_{1<\lambda_B<\lambda_A}\PH_{A,B}^0$ is the set of partially hyperbolic diffeomorphisms $f$ of $\TT^4$ such that $f$ belongs to some $\PH_{A,B}^0$ and there exists a fixed point $p_f=f(p_f)\in\TT^4$, such that $f\mid _{\cF^c_{f}(p_f)}$ is an Anosov diffeomorphism, where $\cF^c_{f}(p_f)$ is the (fixed) center leaf passing through $p_f$. Also let
$$
\mathcal {SH}^r:=\{f\in\mathcal {SH}:\ f \hbox{ is } C^r\},\ \ r\geq 1.
$$
\end{definition}

Although this part will not be used in the proof, through analyzing the induced map on the fundamental group, it is easy to show that $f\mid _{\cF^c_f(p_f)}$ is topological conjugate to $B$. Shub proved the following.

\begin{theorem}\label{Shub}[Shub \cite{Sh}]
$\mathcal {SH}$ is $C^1$ open and every $f\in \mathcal {SH}$ is transitive.
\end{theorem}

Shub proved this result for some specific examples, but the proof can be adapted for the Shub class of diffeomorphisms with minor modifications. In this article we consider the class of Shub diffeomorphisms which also satisfy some bunching conditions.

\begin{definition}
The \emph{bunched Shub class} $\mathcal {SH}^r_b$ is the set of partially hyperbolic diffeomorphisms $f\in\mathcal {SH}^r$ which also satisfy the following bunching conditions:
\begin{enumerate}
\item[(a)]
Global bunching:
\begin{equation}\label{global}
\lambda_s(x)<\frac{\lambda_c^-(x)}{\lambda_c^+(x)}\leq\frac{\lambda_c^+(x)}{\lambda_c^-(x)}<\lambda_u(x),\ \ \forall x\in\TT^4;
\end{equation}
\item[(b)]
Stronger local bunching at the fixed center leaf $\cF^c_f(p_f)$:
\begin{equation}\label{local1}
\lambda_s(x)<\left(\lambda_c^-(x)\right)^2\leq\left(\lambda_c^+(x)\right)^2<\lambda_u(x),\ \ \forall x\in\cF^c_f(p_f),
\end{equation}
and
\begin{equation}\label{local2}
\lambda_s(x)<\frac{\lambda_c^-(x)}{\left(\lambda_c^+(x)\right)^2}\leq\frac{\lambda_c^+(x)}{\left(\lambda_c^-(x)\right)^2}<\lambda_u(x),\ \ \forall x\in\cF^c_f(p_f).
\end{equation}
\end{enumerate}
\end{definition}

Clearly $\mathcal {SH}^r_b$ is a $C^1$ open set.

\begin{remark}\label{bunch}
The condition (\ref{global}) implies (see \cite{PSW}) that if $f$ is $C^2$ then the stable and unstable bundles are $C^1$ when restricted to the center-stable and center-unstable leaves, and as a consequence the strong stable and strong unstable holonomies between the center leaves are of class $C^1$ (when restricted to the center-stable respectively center-unstable leaves). We will see later that in fact these holonomies depend continuously in the $C^1$ topology with respect to the points (or the center leaves) and with respect to the map $f$ (in the $C^2$ topology).
\end{remark}

\begin{remark}
The condition  (\ref{local1}) is the standard 2-bunching condition, and \cite{HPS} implies that if $f$ is $C^2$ then $\cF^c_f(p_f)$, $\cF^{cs}_f(p_f)$ and $\cF^{cu}_f(p_f)$ are of class $C^2$. If the central bounds are symmetric, of $\lambda_c^-\lambda_c^+=1$, then it is equivalent to the global bunching condition (\ref{global}).

The condition (\ref{local2}) implies that if $f$ is $C^3$ then the strong stable and strong unstable bundles and (sub)foliations are of class $C^2$ when restricted inside the center-stable respectively center-unstable leaves of the fixed point $p_f$ (see \cite{PSW}).
\end{remark}

\subsection{Results}

The homoclinic intersections between hyperbolic periodic points were first observed by Poincar\'e, and since then, they play an important role in the theory of dynamical systems. Smale \cite{Sm} used them in order to define homoclinic classes.

\begin{definition}
Given two hyperbolic periodic points $p, q$ of the diffeomorphism $f$, with the same stable index, we say that they are \emph{homoclinically related} if their stable and unstable manifolds intersect transversally:
\begin{equation}
W^s(p)\pitchfork W^u(q)\neq \emptyset \;\;\text{ and }\;\; W^s(q)\pitchfork W^u(p)\neq \emptyset.
\end{equation}We say that $\Orb(p)$ and $\Orb(q)$ are \emph{homoclinically related} if:
\begin{equation}\label{eq.homoclinic}
W^s(\Orb(p))\pitchfork W^u(\Orb(q))\neq \emptyset \;\;\text{ and }\;\; W^s(\Orb(q))\pitchfork W^u(\Orb(p))\neq \emptyset.
\end{equation}
This is an equivalence relation between hyperbolic periodic orbits. The \emph{homoclinic class of $\Orb(p)$}, $HC(\Orb(p))$, is the closure of the equivalence class of $\Orb(p)$.
\end{definition}

If a diffeomorphism has no dominated splitting of index 2, it seems unexpected that any two hyperbolic points of stable index $2$ are homoclinically related to each other.
Indeed, the sizes of stable and unstable manifolds of the hyperbolic periodic points are nonuniform, and the intersection in \eqref{eq.homoclinic}
can be empty. On the other hand, even if the intersection is non-empty, the intersection may not be transverse, because of the lack of domination (see Pujals-Sambarino \cite{PS}, Lan Wen \cite{W}).

The main result of this paper is the following.

\begin{main}\label{main.homoclinic intersection}
For any $2\leq r\leq\infty$, there exists a $C^2$ open and $C^r$ dense subset $\cU^r\subset\mathcal {SH}^r_b$, such that for any $f\in\cU^r$, any two hyperbolic points of $f$ with stable index $2$ are homoclinically related.
\end{main}

As a consequence, any diffeomorphism $f\in \cU^r$ admits a unique homoclinic class associated to the hyperbolic periodic points of index $2$. Denote by $p_f$ a hyperbolic fixed point of $f\in \cU^r$.

\begin{maincor}\label{maincor.suppporthomoclinic}
For any $f\in \cU^r$, $f$ admits a unique homoclinic class $H(p_f,f)$ associate to the hyperbolic periodic points of index $2$, and the homoclinic class coincides to the ambient manifold.
\end{maincor}

For a continuous potential $\phi$ and a continuous map $f$, an $f$-invariant probability measure $\mu$ is called an \emph{equilibrium measure} for the potential $\phi$, if
$$h_\mu(f) + \int \phi d\mu = P_{top}(\phi),$$
where $P_{top}(\phi) := \sup_{\nu\in \cM_e(f)}\{h_\nu(f) + \int \phi d\nu\}$.

The equilibrium states do not necessarily exist. Assuming entropy expansiveness, Bowen \cite{Bo} proved the equilibrium states do exist. It was shown by Liao, Viana and Yang \cite{LVY} that any diffeomorphism away from homoclinic tangencies is entropy expansive. Yomdin \cite{Y}, (see also Buzzi \cite{Bu}) proved also that for any $C^\infty$ diffeomorphism, equilibrium states always exist.

The uniqueness of equilibrium states is a more subtle problem. Recently, Climenhaga and Thompson \cite{CT} (see also \cite{PYY}) gave a criterion based on Bowen property and specification. Another method used by Buzzi, Crovisier and Sarig \cite{BCS} (see also Ben Ovadia \cite{Be1,Be2}) is based on the use of the homoclinic class of measures:

\begin{definition}\label{df.homoclinic measures}
Suppose $f$ is a $C^r$ diffeomorphism for some $r>1$. For two ergodic hypperbolic measures $\mu_1$ and $\mu_2$ of $f$, we write $\mu_1 \preceq \mu_2$ iff there exist measurable sets $A_1, A_2 \subset M$ with $\mu_i(A_i) > 0$ such that for any $x_1\in A_1$ and $x_2\in A_2$, the manifolds $W^u(x_1)$ and $W^s
(x_2)$ have a point of transverse intersection.

$\mu_1$, $\mu_2$ are \emph{homoclinically related} if $\mu_1 \preceq \mu_2$ and $\mu_2\preceq \mu_1$. We write $\mu_1 \overset{h}{\sim}
\mu_2$. The set of ergodic measures homoclinically related to a hyperbolic ergodic measure $\mu$ is called the \emph{measured homoclinic class} of $\mu$.
\end{definition}

\begin{remark}\label{rk.homoclinic measures}
The homoclinic relation is an equivalence relation, moreover, two atomic measures supported on two periodic orbits are homoclinically related
if and only if the two periodic orbits are hyperbolic and homoclinically related.
\end{remark}

We have the following Theorem.

\begin{main}\label{main.measured homoclinic }
For any $f\in \cU^r$, all the hyperbolic ergodic measures of index $2$ are homoclinically related. Let $\phi:\TT^4\rightarrow\mathbb R$ be any H\"older potential function with
$\max_{x,y\in \TT^4}\|\phi(x)-\phi(y)\|<\log \lambda_B$, where $f\in\PH_{A,B}^0$, then $f$ admits at most one equilibrium state for the potential $\phi$. In particular, every $f\in \cU^r$ admits at most one measure of maximal entropy.
\end{main}

A direct consequence of \cite{Y,Bu} is the following.

\begin{corollary}
Every $f\in\cU^{\infty}\cap\PH_{A,B}^0$ admits a unique equilibrium state for every H\"older potential satisfying $\max_{x,y\in \TT^4}\|\phi(x)-\phi(y)\|<\log \lambda_B$. In particular every $f\in\cU^{\infty}$ admits a unique measure of maximal entropy.
\end{corollary}

For Shub's example, some similar results were obtained under some extra assumptions. For instance, by Newhouse, Young \cite{NY} and Carvalho, P\'erez \cite{CP} with the extra assumption that within the center foliation there exists a one-dimensional invariant sub-foliation, and by \'{A}lvarez \cite{A} assuming that the center bundle admits a further dominated splitting.

\section{Preliminaries}

\subsection{Stable and unstable holonomies between center leaves.}

As we mentioned before, the condition \ref{global} implies that the holonomies between the center leafs are uniformly $C^1$. In fact there exists a continuity of the holonomies in the $C^1$ topology. If $y\in\cF^u_f(x)$ let us denote by $h_{f,x,y}^u:\cF^c_f(x)\rightarrow\cF^c_f(y)$ the unstable holonomy between the two center leaves. Since it is of class $C^1$, the derivative $Dh_{f,x,y}^u$ induces a continuous map between the unit tangent bundles $Dh_{f,x,y*}^u:T^1\cF^c_f(x)\rightarrow T^1\cF^c_f(y)$.

\begin{lemma}\label{le:contc1}
$Dh_{f,x,y*}^u$ is continuous with respect to $f\in \mathcal{SH}^2_b$ (the $C^2$ topology) and $(x,y)\in\TT^4,\ \ y\in\cF^u_f(x)$. The same holds for the stable holonomy.
\end{lemma}

\begin{remark}
The continuity in Lemma \ref{le:contc1} means that if $f_n$ converges to $f$ in the $C^2$ topology, $x_n$ converges to $x$ in $\TT^2$, and $y_n\in\cF^u_{f_n, loc}(x_n)$ converges to $y$, then $Dh_{f_n,x_n,y_n*}^u$ converges uniformly to $Dh_{f,x,y*}^u$. The proof requires only the weaker global condition \ref{global}.
\end{remark}

\begin{remark}
Since in our case the stable and unstable bundles are one-dimensional, one could approach the continuity question using the classical ODE theory of the regularity of solutions with respect to the initial conditions and parameters. We prefer to present a different proof which constructs the projectivized holonomies as unstable foliations of the projectivization of $f$ along the center bundle.
\end{remark}

\begin{proof}
Let $T^1\TT^4$ be the unit tangent bundle of $\TT^4$ (which can be identified with $\TT^4\times\mathbb S^3$) with $Df_*$ being the $C^1$ diffeomorphism induced by $f$. We will consider the central unit tangent bundle $S_f:=\cup_{x\in\TT^4}S(f,x)$ where $S(f,x)=T_x^1\cF^c_f(x)$ is the unit circle in $E^c_f(x)$. Then $S_f$ is a H\"older submanifold of $T^1\TT^4$ invariant under $Df_*$, which is also a H\"older bundle over $\TT^4$.

We claim that there exists a continuous unstable foliation on $S_f$ and that $Dh_{f,x,y*}^u$ is exactly the unstable holonomy for this foliation between the transversals $T^1\cF^c_f(x)$ and $T^1\cF^c_f(y)$. We apply the standard construction of the local unstable leaves as the invariant section of a bundle contraction map (see \cite{HPS} for example), with a minor difficulty arising from the lack of smoothness.

For any $x,y\in \TT^4$ we define the $\pi_{f,y,x}:E^c_f(y)\rightarrow E^c_f(x)$ the projection parallel to $E^s_f(x)\oplus E^u_f(x)$. The maps $\pi_{f,y,x}$ depend continuously on $x,y\in\TT^4$ and $f$ (in the $C^1$ topology). For $x$ close to $y$ this is invertible and close to the identity, and its projectivization $\pi_{f,x,y*}$ is bi-Lipschitz with Lipschitz constant close to 1. 

For $\delta>0$ and $x\in\TT^4$ let $\alpha_{f,x}:[-\delta,\delta]\rightarrow\cF^u_f(x)$ be the length parametrization of the local unstable manifold of $f$ at $x$. Since the unstable foliation is orientable and depends continuously in the $C^1$ topology with respect to $x$ and $f$, we have that $\alpha_{f,x}$ is continuous in $x$ and $f$ (in the $C^1$ topology). 

For any $\delta>0$ there exists $\epsilon_{\delta}>0$ such that for any $x,y$ such that $d(x,y)<\delta$ we have:
\begin{itemize}
\item
$\|\pi_{f,x,y}^{\pm 1}-Id\|<\epsilon_{\delta}$,
\item
$\pi_{f,x,y*}^{\pm 1}$ is bi-Lipschitz with constant $(1+\epsilon_{\delta})$,
\item
$(1+\epsilon_{\delta})^{-1}\lambda_c^-(x)<\lambda_c^-(y)\leq\lambda_c^+(y)<(1+\epsilon_{\delta})\lambda_c^+(x)$,
\item
If furthermore $y\in\cF^u_{f,\delta}(x)$ then $d_u(f(x),f(y))\geq(1+\epsilon_{\delta})^{-1}\lambda_u(x)d_u(x,y)$, where $d_u$ is the distance along the unstable leaves.
\end{itemize}
We can choose $\epsilon_{\delta}$ independent of $f$ in a $C^1$ neighborhood and $\lim_{\delta\rightarrow 0}\epsilon_{\delta}=0$.

Now we will construct the bundle with the candidates for the local unstable manifolds in $S_f$. Consider $\delta>0$ (small) to be specified later. Let
$$
B=\{\sigma:[-\delta,\delta])\rightarrow \mathbb R:\ \sigma(0)=0,\ \left|\frac{\sigma(t)}{t}\right|<\infty\}.
$$
the Banach space of functions $\sigma$ bounded for the norm
$$
\|\sigma\|=\sup_{t\in [-\delta,\delta]}\left|\frac{\sigma(t)}{t}\right|.
$$
Then
$$
V(f):=S_f\times B
$$
is a continuous (in fact H\"older) Banach bundle over $S_f$.
\begin{remark}
The maps $\sigma$ are candidates for unstable manifolds in $S_f$ in the following sense. For any $(x,v)\in S_f$ and $\sigma\in B$ we can define a section $\tilde\sigma:\cF^u_{f,\delta}(x)\rightarrow S_f$ in the following way:
$$
\tilde\sigma(y):=\pi_{f,y,x*}^{-1}(v+\sigma(\alpha_{f,x}^{-1}(y)))\in S(f,y).
$$
The graph of this section $\tilde\sigma$ is a natural candidate for the local unstable manifold in $(x,v)\in S_f$. We construct it as a fixed point of the natural graph transformation.
\end{remark}

Let $T:V(f)\rightarrow V(f)$ be the bundle map which fibers over $Df_*$ on $S_f$ and is given by
\begin{eqnarray*}
T\sigma_{(x,v)}(t)&=&\left(\pi_{f,f(y(t)),f(x)}\circ Df(y(t))\circ\pi_{f,y(t),x}^{-1}\right)_*(v+\sigma(\alpha_{f,x}^{-1}(y(t)))-Df(x)_*(v),\\
y(t)&=&f^{-1}\circ\alpha_{f,f(x)}(t).
\end{eqnarray*}
One can check that in fact $T$ is defined in such a way so that we have $\tilde{T\sigma}=Df_*\tilde\sigma$. Let us check that $T$ is a continuous bundle map on $V(f)$, which is also a fiber contraction.

{\bf Claim 1: If $\sigma\in B$ then $T\sigma_{(x,v)}\in B$.}

Remember that $y(t)=f^{-1}\circ\alpha_{f,f(x)}(t)$, and let us denote
$$
G(t):=\left(\pi_{f,f(y(t)),f(x)}\circ Df(y(t))\circ\pi_{f,x,y}^{-1}\right).
$$
Observe that $G(t)_*$ is Lispchitz with the Lipschitz constant
$$
Lip(G(t)_*)=(1+\epsilon_{\delta})^2\frac{\lambda_c^+(y(t))}{\lambda_c^-(y(t))}\leq(1+\epsilon_{\delta})^4\frac{\lambda_c^+(x)}{\lambda_c^-(x)}.
$$
Also
$$
|\alpha_{f,x}^{-1}(y(t))|=d_u(x,y(t))\leq(1+\epsilon_{\delta})\lambda_u(x)^{-1}d_u(f(x),f(y(t))=(1+\epsilon_{\delta})\lambda_u(x)^{-1}|t|.
$$
Then
\begin{eqnarray*}
\|T\sigma\|&=&\sup_{t\in [-\delta.\delta]}\left|\frac{G(t)_*(v+\sigma(\alpha_{f,x}^{-1}(y(t)))-Df(x)_*(v)}t\right|\\
&\leq& \sup_{t\in [-\delta.\delta]}\left|\frac{G(t)_*(v+\sigma(\alpha_{f,x}^{-1}(y(t)))-G(t)_*(v)}t\right|+ \sup_{t\in [-\delta.\delta]}\left|\frac{G(t)_*(v)-Df(x)_*(v)}t\right|\\
&\leq& (1+\epsilon_{\delta})^4\frac{\lambda_c^+(x)}{\lambda_c^-(x)}\sup_{t\in [-\delta.\delta]}\left|\frac{\sigma(\alpha_{f,x}^{-1}(y(t))}t\right|+\frac {\pi}2\sup_{t\in [-\delta.\delta]}\frac 1t\left\|\frac{G(t)(v)}{\|G(t)(v)\|}-\frac{Df(x)(v)}{\|Df(x)(v)\|}\right\|\\
&\leq&(1+\epsilon_{\delta})^5\frac{\lambda_c^+(x)}{\lambda_c^-(x)\lambda_u(x)}\|\sigma\|+\frac{\pi}{\lambda_c^+(x)}\sup_{t\in [-\delta.\delta]}\left\|\frac{G(t)(v)-Df(x)(v)}t\right\|,
\end{eqnarray*}
where in the last line we used the inequality
$$
\left\|\frac a{\|a\|}-\frac b{\|b\|}\right\|\leq\left\|\frac a{\|a\|}-\frac a{\|b\|}\right\|+\left\|\frac a{\|b\|}-\frac b{\|b\|}\right\|\leq\frac 2{\|b\|}\|a-b\|.
$$
Let us remark that if $v\in E^c_f(x)$ then $\pi_{f,f(y(t)),f(x)}\circ Df(x)\circ\pi_{f,y,x}^{-1}(v)=Df(x)(v)$, because the partially hyperbolic splitting is invariant under $Df$. Then
\begin{eqnarray*}
\|G(t)(v)-Df(x)(v)\|&=&\|\pi_{f,f(y(t)),f(x)}\circ (Df(y(t))-Df(x))\circ\pi_{f,y,x}^{-1}(v)\|\\
&\leq&(1+\epsilon_{\delta})^2Lip(Df)d(y(t),x)\\
&\leq& Lip(Df)\frac{(1+\epsilon_{\delta})^3}{\lambda_u(x)}|t|.
\end{eqnarray*}
Finally we obtain the desired bound:
$$
\|T\sigma\|\leq \frac{(1+\epsilon_{\delta})^5\lambda_c^+(x)}{\lambda_c^-(x)\lambda_u(x)}\|\sigma\| + \frac{Lip(Df)(1+\epsilon_{\delta})^3\pi}{\lambda_c^+(x)\lambda_u(x)}.
$$

{\bf Claim 2: $T$ is a fiber contraction.}

\begin{eqnarray*}
\|T\sigma_1-T\sigma_2\| & = & \sup_{t\in [-\delta,\delta]}\left|\frac{T\sigma_1(t)-T\sigma_2(t)}{t}\right|\\
&=& \sup_{t\in [-\delta,\delta]}\left|\frac{G(t)_*(v+\sigma_1(\alpha_{f,x}^{-1}(y(t)))-G(t)_*(v+\sigma_2(\alpha_{f,x}^{-1}(y(t)))}{t}\right|\\
&\leq& (1+\epsilon_{\delta})^4\frac{\lambda_c^+(x)}{\lambda_c^-(x)} \sup_{t\in [-\delta,\delta]}\left|\frac{\sigma_1(\alpha_{f,x}^{-1}(y(t))-\sigma_2(\alpha_{f,x}^{-1}(y(t))}{t}\right|\\
&\leq& (1+\epsilon_{\delta})^5\frac{\lambda_c^+(x)}{\lambda_c^-(x)\lambda_u(x)}\|\sigma_1-\sigma_2\|.
\end{eqnarray*}

Now all we have to do is to choose $\delta$ small enough so that $\epsilon_{\delta}$ is close enough to zero so that we have
$$
(1+\epsilon_{\delta})^5\frac{\lambda_c^+(x)}{\lambda_c^-(x)\lambda_u(x)}<1.
$$

Claims 1 and 2 show that we are in the conditions of the Invariant Section Theorem from \cite{HPS}, so there exists a unique bounded continuous invariant section.

From \cite{PSW} we know that the unstable holonomy along center leaves is uniformly differentiable. The projectivisation of the derivative of this local holonomy will then correspond to a bounded invariant section for the transfer operator $T$, so it has to coincide with the unique continuous invariant section constructed above. This concludes the proof of the continuity of $Dh^u_{f,x,y*}$ with respect to the points $x,y$ (we proved it for $d(x,y)<\delta$, but this can be easily extended to larger distances). 

If $f_n$ converges to $f$ then $S_{f_n}$ converges to $S_f$ (this can be made explicit by projecting $E_{f_n}^c$ to $E^c_f$ parallel to $E_f^s\oplus E_f^u$ for example). One can check that the corresponding transfer operators $T_{f_n}$ also converge to $T_f$. Since the invariant section is continuous with respect to the bundle map, we obtain the continuity of $Dh^u_{f,x,y*}$ with respect to $f$.

\end{proof}

\begin{remark}
We gave the proof for our special setting, but the proof can be adapted to general partially hyperbolic deffeomorphisms in higher dimensions. We used that $Df$ is Lipschitz in order to show that the transfer operator $T$ verifies the conditions of the Invariant Section Theorem. The proof can be adapted for $f$ of class $C^{1+\alpha}$ and the stronger bunching condition $\lambda_s(x)^{\alpha}<\frac{\lambda_c^-(x)}{\lambda_c^+(x)}\leq\frac{\lambda_c^+(x)}{\lambda_c^-(x)}<\lambda_u(x)^{\alpha}$, using the norm $\|\sigma\|=\sup_{t\in[-\delta,\delta]}\left|\frac{\sigma(t)}{t^{\alpha}}\right|$. Once one obtains the bounded invariant section for the projectivization $Df_*$ on $S_f$, using the boundness of the Jacobian, one could try to obtain the differentiability of the stable/unstable holonomy along center leaves.
\end{remark}

\subsection{Homoclinic holonomies}

Let $f\in \mathcal {SH}^2_b$ and $p_f$ the fixed point such that $f\mid_{\cF^c_f(p_f)}$ is Anosov. We will drop the index $f$ when it is not necessary to specify the dependence on the map $f$. From \cite{HPS} and the bunching conditions we know that $\cF^c(p), \cF^{cu}(p)$ and $\cF^{cs}(p)$ are $C^2$ submanifolds. Assume that $q$ is a homoclinic point of $W^c(p)$, i.e. $q\in\cF^{cu}(p)\cap\cF^{cs}(p)$, then $W^c(q)$ is also $C^2$ as a connected component of the intersection of the transverse $C^2$ submanifolds $\cF^{cu}(p)$ and $\cF^{cs}(p)$. We can define the stable holonomy $h_{p,q}^s:\cF^c(p)\rightarrow \cF^c(q)$ and the unstable holonomy $h_{q,p}^u:\cF^c(q)\rightarrow \cF^c(p)$, and they are both of class $C^1$. Then $\tilde h_q:=h_{q,p}^u\circ h_{p,q}^s:\cF^c(p)\rightarrow \cF^c(p)$ is a $C^1$ diffeomorphism, so it induces a $C^0$ map on the unit tangent bundle $T^1\cF^c(p)$ which we denote by $D\tilde h_{q*}$.

Let $\tilde v^s(x)$ be the unit vector tangent in $x\in \cF^c(p)$ to the stable bundle of $f\mid_{\cF^c(p)}$ (we fix an orientation). Since $f\mid_{\cF^c(p)}$ is a $C^2$ Anosov map on a $C^2$ surface, we have that $\tilde v^s:\cF^c(p)\rightarrow T^1\cF^c(p)$ is $C^1$. We define the map $\tilde g_{q}:\cF^c(p)\rightarrow T^1\cF^c(p)$, 
\begin{equation}
\tilde g_{q}(x):=D\tilde h_{q*}(\tilde v^s(h_q^{-1}(x)))=\frac{D\tilde h_q(h_q^{-1}(x))\tilde v^s(h_q^{-1}(x))}{\|D\tilde h_q(h_q^{-1}(x))\tilde v^s(h_q^{-1}(x))\|}.
\end{equation}

\begin{remark}
In fact we consider the stable foliation of $f\mid_{\cF^c(p)}$ inside the leaf $\cF^c(p)$, we first push it forward using the stable holonomy $h_{p,q}^s$ to the leaf $\cF^c(q)$, and then we push it again using the unstable holonomy $h_{q,p}^u$ back to the leaf $\cF^c(p)$. Then $\tilde g_{q}(x)$ is in fact the unit tangent vector in $x$ to this new foliation.
\end{remark}

If furthermore $f\in\mathcal {SH}^3_b$ then the stable and unstable holonomies along the fixed center-stable leaf $W^{cs}(p)$ and respectively the fixed center-unstable leaf $W^{cu}(p)$ are $C^2$, so in this case $D\tilde h_{q*}$ and $\tilde g_q$ are in fact $C^1$.

If the map $f'$ is $C^2$ close to $f$, then the fixed Anosov center leaf $\cF^c(p)$ and its homoclinic center leaf $\cF^c(q)$ will have continuations $\cF^c_{f'}(p(f'))$ and $\cF^c_{f'}(q(f'))$. Then we obtain the continuations of the stable holonomy $h_{p(f'),q(f'),f'}^s:\cF^c_{f'}(p(f'))\rightarrow \cF^c_{f'}(q(f'))$ and the unstable holonomy $h_{q(f'),p(f'),f'}^u:\cF^c_{f'}(q(f'))\rightarrow \cF^c_{f'}(p(f'))$; they are $C^1$ maps and depend continuously in the $C^1$ topology with respect to $f'$ (in the $C^2$ topology). We also have a continuation of the homoclinic holonomy $\tilde h_{q(f'),f'}:\cF^c_{f'}(p(f'))\rightarrow \cF^c_{f'}(p(f'))$ and also the continuation $\tilde g_{q(f'),f'}:\cF^c_{f'}(p(f'))\rightarrow T^1\cF^c_{f'}(p(f'))$, which is continuous both with respect to $x\in \cF^c_{f'}(p(f'))$ and with respect to $f'\in\mathcal {SH}^2_b$ (in the $C^2$ topology).

\subsection{Hyperbolic measures\label{ss.hyperbolic measures}}

Let $\mu$ be an ergodic measure of a diffeomorphism $f$, then by the Theorem of Oseledets, for $\mu$-almost every point $x\in M$, there exist $k(\mu)\in \mathbb{N}$, real numbers $\lambda_1(\mu)>\cdots \lambda_k(\mu)$ and an invariant splitting $T_xM=E^1(x)\oplus \cdots \oplus E^k(x)$ of the tangent bundle at $x$, depending measurably on the point, such that $\lim_{n\to \pm \infty} \frac{1}{n}\log \|Df^n_x(v)\|=\lambda_j(\mu)$ for all $0\neq v\in E^j(x)$. The real numbers $\lambda_j(\mu)$ are the \emph{Lyapunov exponents} of $\mu$. We say that the ergodic measure $\mu$ is \emph{hyperbolic} if all the Lyapunov exponents of $\mu$ are non-zero.

\begin{theorem}\label{l.closing}[Katok's Horseshoe Theorem \cite{K}:]
For any $f\in \Diff^r(M)$, $r>1$ and any hyperbolic ergodic measure $\mu$, there exists a hyperbolic periodic point $p$, such that $\mu \overset{h}{\sim} \delta_{\Orb(p)}$,
where $\delta_{\Orb(p)}$ is the ergodic measure supported on the orbit $\Orb(p)$.
\end{theorem}

If a diffeomorphism $f$ admits a dominated splitting, then the Oseledet's splitting must be subordinated to the dominated splitting. In particular, since every $f\in \mathcal{SH}$ is partially hyperbolic on $\mathbb{T}^4$, then for any ergodic measure $\mu$ of $f$, its biggest Lyapunov exponent is positive ($\lambda^u>0$) and its associated Oseledet's bundle is tangent to the strong unstable bundle $E^u$ of $f$. A similar result holds for the minimal Lyapunov exponent $\lambda^s<0$ with its associated Oseledet's bundle tangent to the strong stable bundle $E^s$. There are also two center Lyapunov exponents (counted with multiplies) $\lambda^c_1\geq \lambda^c_2$ whose associated Oseledet's bundles are tangent to the center bundle $E^c$ of $f$.

\subsection{Criterion of uniqueness of equilibrium state}

\begin{definition}
Let $\mu$ be an ergodic hyperbolic measure of a diffeomorphism $f$. The \emph{stable index of $\mu$} is the number of negative Lyapunov exponents, counted with multiplicity.
\end{definition}

\begin{proposition}\label{p.criterion}
Let $f: M \to M$ be a $C^r$ diffeomorphism $r > 1$, $\phi:M\rightarrow\mathbb R$ be a H\"older potential, and $p$ a hyperbolic periodic point. Then there is at most one equilibrium state for $\phi$ which is homoclinically related to $\delta_{Orb(p)}$, and its support coincides with $\HC(Orb(p))$.
\end{proposition}

\begin{proof}
This is explained in \cite{Be2}[Theroem 2.4] and \cite{BCS}[Section 1.6]. See also \cite{Be1} and \cite{BCS}[Corollary 3.3].
\end{proof}

\subsection{Hyperbolicity of equilibrium states}

If $f\in\PH_{A,B}^0$ then a standard results of Franks-Manning \cite{Fra, Man} imply that $f$ is semi-conjugate to $f_{A,B}$, i.e., there exists a continuous surjection
$h : \mathbb{T}^4\to \mathbb{T}^4$ homotopic to the identity such that $f_{A,B}\circ h=h\circ f$. By Ledrappier-Walters variational principle \cite{LW} we have
\begin{equation}\label{eq.globalentropy}
h_{top}(f)\geq h_{top}(f_{A,B})=\log\lambda_A+\log \lambda_B.
\end{equation}
 
For any invariant probability measure $\mu$ of $f$, we say that a measurable partition $\xi$ is $\mu$ \emph{adapted (sub-ordinated) to $\cF^u$} if the following conditions are satisfied:
\begin{itemize}
\item there is $r_0 > 0$ such that $\xi(x) \subset B^{\cF^u}_{r_0}(x)$ for $\mu$ almost every $x$, where $B^{\cF^u}_{r_0}(x)$ is a ball of $\cF^u(x)$ with radius $r_0$;
\item $\xi(x)$ contains an open neighborhood of $x$ inside $\cF^u(x)$;
\item $\xi$ is increasing, that is, for $\mu$ almost every $x$, $\xi(x)\subset f(\xi(f^{-1}(x)))$.
\end{itemize}
The existence of such a partition was provided by \cite{LS} (see also \cite{LY} and \cite{Ya}). The \emph{partial entropy of $\mu$ along the expanding foliation $\cF^u$} is defined by
$$h_\mu(f, \cF^u) = H_\mu(f^{-1}\xi \mid \xi).$$
The definition of the partial entropy does not depend on the choice of the partition. 

The following two lemmas are important for our further discussion:

\begin{lemma}\label{l.upuentropy}
If $f\in \PH_{A,B}^0$, then $h_\mu(f,\cF^u)\leq \log \lambda_A$.
\end{lemma}

\begin{proof}
Denote by $\cF^c_f$ the center foliation of $f$. By Proposition~\ref{p.center classification},  the projection map $\pi^c_f$ along the center foliation induces a topological
Anosov homeomorphism $\overline{f}$ on the quotient space $\overline{\TT}^2_f=\TT^4/\cF^c_f$, which is topological conjugate to $A$, so we may in fact identify $\overline{\TT}^2_f$ with $\TT^2$ and $\overline{f}$ with $A$.

Denote by $\cF^s_A$ (resp. $\cF^u_A$) the stable (resp. unstable) foliation of $A$. The projection map $\pi^c_f$ maps each center unstable leaf $\cF^{cu}$
of $f$ to an unstable leaf $\cF^u_A$ of $A$. In particular, $\pi^c_f$ maps every unstable leaf $\cF^u$ of $f$ to an unstable leaf $\cF^u_A$ of $A$. Proposition~\ref{p.center classification} implies that all the hypothesis of Tahzibi-Yang~\cite{TY}[Theorem A] are satisfied (see also \cite{A}[Section 2.5]), and this implies that $h_\mu(f,\cF^u)\leq h_{top}(A)=\log \lambda_A$.
\end{proof}

The following lemma is a generalization of \cite{A}:
\begin{lemma}\label{l.largeentropy}
Leb $f\in \PH_{A,B}^0$ be a $C^r$ diffeomorphism, $r>1$, and $\mu$ an ergodic invariant measure of $f$ with $h_\mu(f)> \log \lambda_A$. Then $\mu$ is an hyperbolic ergodic measure of $f$ with stable index $2$, i.e., $\lambda^c_1>0>\lambda^c_2$.
\end{lemma}

\begin{proof}
We will show that $\lambda^c_1>0$. To prove that $\lambda^c_2<0$, one only needs to consider the diffeomorphism $f^{-1}$ instead of diffeomorphism $f$.

Suppose by contradiction that $\lambda^c_1\leq 0$. The entropy formula of Ledrappier-Young (see \cite{LS}, \cite{Br}) implies that $h_\mu(f)=h_\mu(f,\cF^u)$. 

Combining with the previous lemma, we obtain that $h_\mu(f)\leq \log \lambda_A$, which is a contradiction with the hypothesis that $h_\mu(f)>\log \lambda_A$. The proof is complete. 
\end{proof}

\section{Proof of theorem \ref{main.homoclinic intersection}}

\subsection{Definition of $\mathcal U$ and plan of the proof}

\begin{definition}
Let $\mathcal U_s^r\subset \mathcal {SH}^r_b$,
\begin{equation}
\mathcal U_s^r=\left\{f\in \mathcal {SH}^r_b:\ \forall x\in \cF^c(p),\exists q \hbox{ homoclinic to $\cF^c(p)$ such that } \tilde g_{q}(x)\neq \pm\tilde v^s(x)\right\}.
\end{equation}
In a similar way we define $\mathcal U_u^r$. Let $\mathcal U^r=\mathcal U_s^r\cap\mathcal U_u^r$.
\end{definition}

The proof of Theorem \ref{main.homoclinic intersection} is divided in the following 3 propositions.

\begin{proposition}\label{prop:open}
$\mathcal U^r$ is $C^2$ open.
\end{proposition}

\begin{proof}

An immediate consequence of the compactness of $\cF^c(p)$ and of the fact that the stable and unstable holonomies depend continuously in the $C^1$ topology with respect to the points (see Remark \ref{bunch}) is the following lemma.

\begin{lemma}\label{le:finite}
Let $f\in\mathcal {SH}^r_b$. Then $f\in\mathcal U_s^r$ if and only if there exist $q_1,q_2,\dots q_k$ homoclinic points of $\cF^c(p)$ such that the image of $\tilde g_{q_1}\times \tilde g_{q_2}\times\dots\times \tilde g_{q_k}$ is disjoint from the image of $\pm\tilde v^{s^k}$.
\end{lemma}

On the other hand, the holonomies along the center leaves depend continuous in the $C^1$ topology with respect to the map $f$ (in $C^2$ topology), so the images of $\tilde g_{q_1}\times \tilde g_{q_2}\times\dots\times \tilde g_{q_k}$ and $\pm\tilde v^{s^k}$ depend continuously on the map $f$. Since these images are compact, this concludes the $C^2$ openness of $\mathcal U_s^r$. The proof for $\mathcal U_u^r$ is similar, so $\mathcal U^r=\mathcal U_s^r\cap\mathcal U_u^r$ is $C^2$ open.
\end{proof}

\begin{proposition}\label{prop:dense}
$\mathcal U^r$ is $C^r$ dense.
\end{proposition}

We will give the proof of the proposition in the subsection \ref{ss:density}.

\begin{proposition}\label{prop:related}
If $f\in\mathcal U^r$, then any two hyperbolic periodic points of $f$ of index 2 are homoclinically related.
\end{proposition}

We will give the proof of the proposition in the subsection \ref{ss:hr}

\subsection{Proof of $C^r$ density}\label{ss:density}

We will show that $\mathcal U_s^r$ is $C^r$ dense in $\mathcal {SH}^r_b$, the proof for $\mathcal U_u^r$ is similar. Then $\mathcal U^r$ will be $C^r$ dense as the intersection of two $C^r$ open dense sets.

The main perturbation result which we will use is the following Lemma.

\begin{lemma}\label{le:perturbation}
Let $f\in\mathcal {SH}^3_b$. Let $q$ be a homoclinic point of the fixed Anosov leaf $\cF^c(p)$. Then
\begin{enumerate}
\item
For any $C^2$ family of perturbations $(\phi_T)_{T\in\mathbb R^n}$ of the identity on $\mathbb T^4$ ($\phi_{0^n}=Id_{\mathbb T^4}$), supported in a neighborhood of $\cF^c(q)$ disjoint from all the other iterates $f^k(\cF^c(q)),\ k\in\mathbb Z\setminus\{0\}$, the map  $(T,x)\mapsto \tilde g_{q(f\circ\phi_T),f\circ\phi_T}(x)$ is $C^1$ on $[-\delta,\delta]^n\times \cF^c(p)$, for some $\delta>0$.
\item
For any $x_0\in \cF^c(p)$ and any $r_0>0$ there exists a $C^{\infty}$ family $(\phi_t)_{t\in[-\delta,\delta]}$ of (volume-preserving) perturbations of the identity on $\mathbb T^4$, supported in $B(y_0,r_0)$ where $y_0:=h_{p,q}^u(x_0)\in \cF^c(q)$, such that 
\begin{equation}\label{eq:dnezero}
\frac{\partial}{\partial t}\tilde g_{q(f\circ\phi_t),f\circ\phi_t}(x)\mid_{(x,t)=(x_0,0)}\neq 0.
\end{equation}
\end{enumerate}
\end{lemma}

\begin{proof}
{\bf Part (1).} Since $\cF^c(p)$ is compact, it is enough to prove that $\tilde g_{q(f\circ\phi_t),f\circ\phi_t}(x)$ is $C^1$ in $(x,T)$ in a small neighborhood of every point $(x_0,0^n)\in \cF^c(p)\times\mathbb R^n$.

Let $x_0\in \cF^c(p)$ and denote $y_0=h_{p,q}^u(x_0)\in \cF^c(q)$, $y_1=f(y_0)\in \cF^c(f(q))$, $z_1=h_{f(q),p}^s(y_1)\in \cF^c(p)$ and $z_0=f^{-1}(z_1)\in \cF^c(p)$. The $f$-invariance of the stable holonomy implies that $h_{q,p}^s(y_0)=z_0$, or $\tilde h_q(z_0)=x_0$.

Let $\psi_p:\mathbb T^2\rightarrow \cF^c(p)$ be a $C^2$ embedding, $a_0=\psi_p^{-1}(x_0),\ b_0=\psi_p^{-1}(z_0)$. Let $I_{\delta}=[-\delta,\delta]$. There exist $C^2$ foliations charts of $\cF^s$ (respectively $\cF^u$) on a small neighborhood of $y_1$ (respectively $y_0$) inside $\cF^{cs}(p)$ (respectively $\cF^{cu}(p)$):
\begin{eqnarray*}
\alpha^s:B^c_{y_1}\times I_{\delta}\rightarrow B^{cs}_{y_1},&\ \alpha^s(\cdot,0)=Id_{\cF^c_{loc}(y_1)},\ \alpha^s(\{y\}\times I_{\delta})=\cF^s_{loc}(y),\ \forall y\in B^c_{y_1};\\
\alpha ^u:B^c_{y_0}\times I_{\delta}\rightarrow B^{cu}_{y_0},&\ \alpha^u(\cdot,0)=Id_{\cF^c_{loc}(y_0)},\ \alpha^u(\{y\}\times I_{\delta})=\cF^u_{loc}(y),\ \forall y\in B^c_{y_0},
\end{eqnarray*}
where $B_x^*$ denotes a small ball centered in $x$ inside $\cF^*(x)$. Define the $C^2$ maps
\begin{eqnarray*}
\beta^s:B_{b_0}\times I_{\delta}\rightarrow B^{cs}_{y_1},&\ \beta^s(b,s)=\alpha^s(h^s_{p,f(q)}(f(\psi_p(b))),s);\\
\beta ^u:B_{a_0}\times I_{\delta}\rightarrow B^{cu}_{y_0},&\ \beta^u(a,r)=\alpha^u(h_{p,q}^u(\psi_p(a)),r),
\end{eqnarray*}
where $B_x$ is a small ball centered at $x$ in $\mathbb T^2$.

We know that the support of $\phi_T$ does not intersect $f^k(\cF^{cs}_{loc}(y_1)), \forall k\geq 0$ and $f^l(\cF^{cu}_{loc}(y_0)), \forall l<0$. This implies that $\cF^{cs}_{loc}(y_1)$ remains a local center-stable leaf for $f_T:=f\circ\phi_T$ for all $T$, $\beta^s(a,\cdot)$ remain parametrizations of the strong stable manifolds inside $\cF^{cs}_{loc}(y_i)$, $\cF^{cs}_{loc}(y_1)$ remains inside $\cF^{cs}(p)$ and the stable holonomy between $\cF^c(p)$ and $\cF^{cs}_{loc}(y_1)$ is unchanged. A similar statement holds for $\cF^{cu}_{loc}(y_0)$.

The maps $f_T$ do change the center leaves $\cF^c(q)$, we have that $\cF^c(q(f_T),f_T)=f^{-1}(\cF^s_{loc}(\cF^c(y_1)))\cap \cF^u_{loc}(\cF^c(y_0))$. We can in fact compute implicitly the homoclinic stable-unstable holonomy $\tilde h_{q(f_T), f_T}$ in a neighborhood of $z_0$ in the following way:
\begin{eqnarray*}
\psi_p(a)=\tilde h_{q(f_T),f_T}(\psi_p(b)) & \iff &  h^u_{p,q(f_T),f_T}(\psi_p(a))=h^s_{p,q(f_T),f_T}(\psi_p(b))\\
& \iff & h^u_{p,q(f_T),f_T}(\psi_p(a))=f_T^{-1}(h^s_{p,q(f_T),f_T}(f(\psi_p(b)))\\
&\iff & \beta^u(a,r)=f_T^{-1}(\beta^s(b,s)) \hbox{ for some } r,s\in I_{\delta}\\
&\iff & \phi_T(\beta^u(a,r))=f^{-1}(\beta^s(b,s)) \hbox{ for some } r,s\in I_{\delta}.
\end{eqnarray*}
In conclusion, denoting $h_T=\psi_p^{-1}\circ \tilde h_{q(f_T),f_T}\circ\psi_p$ (the map $\tilde h_{q(f_T),f_T}$ in the chart $\psi_p$) we have
\begin{equation}
a=h_T(b)\iff  \phi_T(\beta^u(a,r))=f^{-1}(\beta^s(b,s)) \hbox{ for some } r,s\in I_{\delta}.
\end{equation}

We choose a $C^{\infty}$ chart $\psi_q:B_{y_0}\rightarrow\mathbb R^4$ (can also be volume preserving) such that
\begin{itemize}
\item $\psi_q(y_0)=0^4$;
\item $D\psi_q(y_0)(E^c(q_0))=sp\{e_1,e_2\}$;
\item $D\psi_q(y_0)(E^u(q_0))=sp\{e_3\}$;
\item $D\psi_q(y_0)(E^s(q_0))=sp\{e_4\}$.
\end{itemize}

Let $E:B_{a_0}\times B_{b_0}\times I_{\delta}^{n+2}\rightarrow\mathbb R^4$,
\begin{equation}
E(a,b,r,s,T)=\psi_q(\phi_T(\beta^u(a,r)))-\psi_q(f^{-1}(\beta^s(b,s))).
\end{equation}
We have that $E$ is $C^2$ and $E(a_0,b_0,0,0,0^n)=\psi_q(y_0)-\psi_q(f^{-1}(y_1))=0$.

{\bf Claim:} $\frac{\partial E}{\partial (b,r,s)}(a_0,b_0,0,0,0^N)$ is invertible.\\
We observe that since $\alpha^s$ is a diffeomorphism such that $\alpha^s(\{y\}\times I_{\delta})=\cF^s_{loc}(y)$, we have that $D\alpha(y,0)\cdot\frac{\partial}{\partial s}$ is a nonzero vector in $E^s(y)$. Since $Df$ preserves $E^s$ and $D\psi_q(y_0)$ takes $E^s(q_0)$ to the line generated by $e_4$, we have that $DE(a_0,b_0,0,0,0^n)$ takes the line generated by $\frac{\partial}{\partial s}$ isomorphically to the line generated by $e_4$.

A similar argument shows that $DE(a_0,b_0,0,0,0^n)$ takes the line generated by $\frac{\partial}{\partial r}$ isomorphically to the line generated by $e_3$ (remember that $\phi_{0^n}=Id_{\mathbb T^4}$).

Now let us analyze the action of $DE(a_0,b_0,0,0,0^n)$ on the two-dimensional space $T_{b_0}B_{b_0}$. It is not hard to see that $D\beta^s(b_0,0)$ takes $T_{b_0}B_{b_0}$ isomorphically to $E^c(y_1)$. Since $Df$ preserves $E^c$ and $D\psi_q(y_0)$ takes $E^c(q_0)$ to the plane generated by $e_1$ and $e_2$, we have that $DE(a_0,b_0,0,0,0^n)$ takes $T_{b_0}B_{b_0}$ isomorphically to the plane generated by $e_1$ and $e_2$. This concludes the proof of the claim.

Now let us finish the proof of the first part of the lemma. The Implicit Function Theorem gives us the existence of a $C^2$ function $H:B_{a_0}\times I_{\delta}^N\rightarrow B_{b_0}\times I_{\delta}^2$, $H(a,T)=(h(a,T), r(a,T),s(a,T))$ such that $E(a,h(a,T),r(a,T),s(a,T),T)=0$ (eventually by making smaller the balls and the intervals). Then the map $h_T(a)=h(a,T)$ is $C^2$ in both variables, which means that $\tilde h_{q(f_T), f_T}(x)$ is $C^2$ in both variables, and then $\tilde g_{q(f\circ\phi_T),f\circ\phi_T}(x)$ is $C^1$ in both variables. This finishes the proof of the first part.

{\bf Part (2).} We will use the same notations from part (1). Let $\rho:\mathbb R^4\rightarrow [0,\infty)$ be a smooth bump function supported on a small ball centered at the origin, and constantly equal to one near the origin. The family $\phi_t:\mathbb T^4\rightarrow\mathbb T^4$ is defined as
$$
\phi_t:=\psi_{q}^{-1}\circ (R_{\rho t}\times Id_{\mathbb R^2})\circ\psi_q,
$$
where $R_t$ is the rotation of angle $t$ in $\mathbb R^2$. Assume that the support of $\rho$ is small enough so that the support of $\phi_t$ is inside $B(y_0,r_0)$ and disjoint of all the other iterates of $W^c(q)$. From part (1) we have
\begin{equation}\label{eq:e}
E(a,b,r,s,t)=(R_{\rho t}\times Id|{\mathbb R^2})(\psi_q(\beta^u(a,r)))-\psi_q(f^{-1}(\beta^s(b,s)))
\end{equation}

We will compute $DE(a_0,b_0,0,0,t)$. Observe that $L_b:=DE(a_0,b_0,0,0,t)\mid_{T_{b_0}B_{b_0}}:T_{b_0}B_{B_0}\rightarrow sp\{e_1,e_2\}$ and $L_s:=DE(a_0,b_0,0,0,t)\mid_{sp\{\frac{\partial}{\partial s}\}}:sp\{\frac\partial{\partial s}\}\rightarrow sp\{e_4\}$ are isomorphisms independent of $t$. Since $D(R_{\rho t}\times Id_{\mathbb R^2})$ keeps $e_3$ invariant we have that also $L_r:=DE(a_0,b_0,0,0,t)\mid_{sp\{\frac{\partial}{\partial r}\}}:sp\{\frac\partial{\partial r}\}\rightarrow sp\{e_3\}$ is also an isomorphism independent of $t$, and
$$
\frac{\partial E}{\partial(b,r,s)}(a_0,b_0,0,0,t)=L_b\times L_r\times L_s.
$$
From \ref{eq:e} we can compute
$$
DE(a_0,b_0,0,0,t)\mid_{T_{a_0}B_{a_0}}=R_t\circ L_a:T_{a_0}B_{a_0}\rightarrow sp\{e_1,e_2\},
$$
where $L_a:=DE(a_0,b_0,0,0,0)\mid_{T_{a_0}B_{a_0}}:T_{a_0}B_{a_0}\rightarrow sp\{e_1,e_2\}$ is an isomorphism. From the Implicit Function Theorem we deduce that
\begin{equation}\label{eq:dht}
Dh_t(a_0)=L_b^{-1}\circ R_t\circ L_a.
\end{equation}

Define $g:B_{a_0}\times I_{\delta}\rightarrow\mathbb T^1$,
$$
g(a,t):=D\psi_p^{-1}(a)_*\tilde g_{q(f\circ\phi_t),f\circ\phi_t}(\psi(a))=\frac{Dh_t(a)(v^s(a))}{\|Dh_t(a)(v^s(a))\|}
$$
where $D\psi_p^{-1}(a)_*$ is the diffeomorphism induced by $D\psi_p^{-1}(a)$ on the unit tangent bundles and $v^s(a)=D\psi_p^{-1}(a)(\tilde v^s(\psi_p(a)))$. In other words, $g(\cdot,t)$ is in fact the map $\tilde g_{q(f\circ\phi_t),f\circ\phi_t}$ seen in the chart $\psi_p$ which identifies $W^c(p)$ with $\mathbb T^2$ and the unit tangent spaces to $W^c(p)$ with $\mathbb T^1$. In order to prove \ref{eq:dnezero} it is enough to show that
$$
\frac{\partial}{\partial t}g(a,t)\mid_{(a,t)=(a_0,0)}\neq 0,
$$
which in turns is equivalent to the fact that $Dh_0(a_0)(v^s(a_0))$ and $\frac{\partial}{\partial t}Dh_t(a_0)(v^s(a_0))\mid_{t=0}$ are not collinear. Using \ref{eq:dht} we obtain $Dh_0(a_0)(v^s(a_0))=L_b^{-1}\circ L_a(v^s(a_0))$ and  $\frac{\partial}{\partial t}Dh_t(a_0)(v^s(a_0))\mid_{t=0}=L_b^{-1}\circ R_{\frac{\pi}2}\circ L_a(v^s(a_0))$, which are clearly non-collinear since $L_a$ and $L_b$ are isomorphisms while $v^s(a_0)$ is non-zero. This finishes the proof of part (2).
\end{proof}

Now let us prove Proposition \ref{prop:dense}.

\begin{proof}[Proof of Proposition \ref{prop:dense}]

Let $f\in\mathcal {SH}^r_b$. We need to find maps in $\mathcal U^r_s$ arbitrarily $C^r$ close to $f$. Since the $C^{\infty}$ maps are dense in the $C^r$ maps in the $C^r$ topology (even inside the volume preserving class), we can assume that $f$ is $C^{\infty}$.

Choose $q_1,q_2,q_3$ homoclinic points of $\cF^c(p)$ such that the orbits of the homoclinic leaves $\cF^c(q_i)$ are mutually disjoint, $i\in\{1,2,3\}$.

For any $x\in \cF^c(p)$ and any $1\leq i\leq 3$, there exists $r_{x,i}>0$ such that if $y_i:=h_{p,q_i}^u(x)\in \cF^c(q_i)$, then the ball $B(y_i,r_{x,i})$ is disjoint from $\cF^c(p)$, from all the iterates $f^k(\cF(q_i)),\forall k\neq 0$, and from all the iterates of $\cF^c(q_j)$, $j\neq i$. Applying Lemma \ref{le:perturbation} part (2) we obtain the family of perturbations $\phi_{t,x,i}$ such that the derivative of  $\tilde g_{q_i,f\circ\phi_{t,x,i}}$ with respect to $t$ in $(x,0)$ does not vanish. By the continuity of the derivative there exists a neighborhood $U_{x,i}$ of $x$ such that $\frac{\partial}{\partial t}\tilde g_{q_i,f\circ\phi_{t,x,i}}$ is nonzero on $\overline U_{x,i}\times\{0\}$.

Let $U_x=\cap_{i=1}^3U_{x,i}$. By compactness of $\cF^c(p)$ there exist finitely many $x^1,x^2,\dots x^K\in \cF^c(p)$ such that $\cF^c(p)=\cup_{j=1}^KU_{x^j}$.

Let us fix some notations. Denote 
$$T=(t_i^j)_{1\leq i\leq 3,\\ 1\leq j\leq K}=(T_i)_{1\leq i\leq 3}=(T^j)_{1\leq j\leq K}\in I^{3K}:=[-\delta,\delta]^{3K},$$
with $T_i=(t_i^j)_{1\leq j\leq K}\in I^K$, $i\in\{1,2,3\}$ and $T^j=(t_i^j)_{1\leq i\leq 3}\in I^3$, $j\in\{1,2,\dots K\}$.

For every $1\leq i\leq 3$ we let $\phi_i:\mathbb T^4\times I^K\rightarrow\mathbb T^4$ given by
\begin{equation}
\phi_i(\cdot, T_i)=\phi_{t_i^1,x^1,i}\circ\phi_{t_i^2,x^2,i}\circ\dots\circ\phi_{t_i^K,x^K,i},\ \ \forall T_i\in I^K.
\end{equation}

We define $\phi,\ F:\mathbb T^4\times I^{3K}\rightarrow\mathbb T^4$, 
\begin{eqnarray*}
\phi(\cdot,T)=\phi_T(\cdot)&:=&\phi_1(\cdot, T_1)\circ\phi_2(\cdot,T_2)\circ \phi_3(\cdot,T_3),\\
F(\cdot,T)=F_T(\cdot)&:=&f\circ\phi_T.
\end{eqnarray*}

The maps $\phi_i$, $\phi$ and $F$ have the following properties:
\begin{enumerate}
\item
$\phi_i$, $\phi$ and $F$ are of class $C^{\infty}$ on $(x,T)$;
\item
$\phi_i$ is a small perturbation of the identity on a small neighborhood of $\cF^c(q_i)$, in particular it leaves the other homoclinic orbits of $\cF^c(q_j)$ unchanged for $j\neq i$;
\item
$F_T$ is equal to $f$ on a neighborhood of $\cF^c(p)$, so it does not change $\cF^c(p)$ and the function $\tilde v^s$;
 \end{enumerate}

Let $V_j=\psi_p^{-1}(U_{x^j})$, where $\psi_p:\mathbb T^2\rightarrow \cF^c(p)$ is the $C^2$ embedding. For every $1\leq i\leq 3$ define $g_i:\mathbb T^2\times I^{3K}\rightarrow\mathbb T^1$,
$$
g_i(x,T)=D\psi_p^{-1}(x)_*\tilde g_{q_i(f_T),f_T}(\psi_p(x)).
$$
In other words, $g_i(\cdot,T)$ is again the map $\tilde g_{q_i(f_T),f_T}$ seen in the chart $\psi_p$ which identifies $\cF^c(p)$ with $\mathbb T^2$ and the unit tangent spaces $T^1\cF^c(p)$ with $\mathbb T^1$. Lemma \ref{le:perturbation} part (1) tells us that $g_i$ is $C^1$ with respect to $(x,T)\in\mathbb T^2\times I^{3K}$ (maybe for a smaller interval $I$). Furthermore
\begin{equation}\label{eq:det1}
\frac{\partial g_i}{\partial t_i^j}(x,T)\neq 0,\ \forall (x,T)\in \overline V_j\times\{0\}^{3K},\ \forall 1\leq j\leq K.
\end{equation}
On the other hand, because for $l\neq i$, the perturbation $\phi_l$ does not touch the orbit of $\cF^c(q_i)$, we have
\begin{equation}\label{eq:det2}
\frac{\partial g_i}{\partial t_l^j}(x,T)= 0,\ \forall (x,T), \forall l\neq i,\ \forall 1\leq j\leq K.
\end{equation}

Define $G:\mathbb T^2\times I^{3K}\rightarrow\mathbb T^3$
\begin{equation}
G(x,T)=(g_1(x,T),g_2(x,T),g_3(x,T)).
\end{equation}

Again $G$ is $C^1$ in $(x,T)\in\mathbb T^2\times I^{3K}$. The formulas \ref{eq:det1} and \ref {eq:det2} tell us that for every $1\leq j\leq K$ we have
$$
\det\left(\frac{\partial G}{\partial T^j}(x,T)\right)=\det\left(\frac {\partial g_i}{\partial t_i^j}(x,T)\right)=\prod_{i=1}^3\frac {\partial g_i}{\partial t_i^j}(x,T)\neq 0,\ \forall (x,T)\in \overline V_j\times\{0\}^{3K}
$$
From the compactness of $\overline V_j$ and the $C^1$ continuity of $G$ with respect to $T$, there exists $0\in J\subset I$ such that, for all $1\leq j\leq K$ we have
\begin{equation}\label{eq:rank}
\det\left(\frac{\partial G}{\partial T^j}(x,T)\right)\neq 0,\ \forall (x,T)\in V_j\times J^{3K},
\end{equation}
and since every point from $\mathbb T^2$ is inside some $V_j$ we conclude that $G$ has maximal rank at every point in $\mathbb T^2\times J^{3K}$.

Remember that $v^s:\mathbb T^2\rightarrow\mathbb T^1$ is the $C^1$ map given by $v^s(x)=D\psi_p^{-1}(x)_*\tilde v^s(\psi(x))$. Let $A:=\{(x,T)\in\mathbb T^2\times J^{3K}:\ G(x,T)\in\{-v^s(x),v^s(x)\}^3\}$  and $B=\pi_2(A)$, where $\pi_2$ is the projection from $\mathbb T^2\times J^{3K}$ on the $T$ component in $J^{3K}$.

A simple consequence of the above definitions is the fact that if $T\notin B$ then $f_T\in \mathcal U_s^r$. In order to finish the proof of the density of $\mathcal U_s^r$ we have to find $T$ arbitrarily close to $0^{NK}$ such that $T\notin B$. We will prove in fact that $B$ has Lebesgue measure zero in $J^{3K}$.

It is enough to show this for $B_1=\pi_2(A_1)$ where $A_1=\{(x,T)\in\mathbb T^2\times J^{3K}:\ G(x,T)=v^s(x)^3\}$, the other combinations of $\pm v^s$ work similarly. Let $H(x, T)=G(x,T)-v^s(x)^3$, this is a $C^1$ map from $\mathbb T^2\times J^{3K}$ to $\mathbb T^3$. Formula \ref{eq:rank} tells us that $H$ has maximal rank equal to 3 at every point ($v^s$ is independent of $T$), so $H^{-1}(0^3)$ is a $C^1$ submanifold of codimension 3 (or dimension $3K-1$) inside $\mathbb T^2\times J^{3K}$. Since $\pi_2\mid_{H^{-1}(0^3)}:H^{-1}(0^3)\rightarrow J^{3K}$ is a $C^1$ map, Sard's Theorem tells us that the image $B_1$ has Lebesgue measure zero.

This implies that we can find arbitrarily small $T\notin B$, which finishes the proof of the $C^r$ density of $\mathcal U_s^r$.

\end{proof}

\subsection{Proof of Proposition \ref{prop:related}}\label{ss:hr}

\begin{proof}
We first remark that, because of the transitivity of the homoclinic relation,  it is enough to show that every hyperbolic periodic point of index 2 of $f\in\mathcal U^r$ is homoclinically related to the fixed point $p$ of the hyperbolic fixed leaf $\cF^c(p)$.

Let $x$ be a hyperbolic point of $f\in\mathcal U^r$ of index 2. Let $\tilde v^u(x)$ be the unit tangent vector to the weak unstable direction inside $T_x\cF^c(x)$. The strong unstable manifold $\cF^u(x)$ must accumulate on the fixed hyperbolic leaf $\cF^c(p)$, so there exists a sequence of homoclinic points $p_n\in \cF^u(x)\cap \cF^s_{loc}(\cF^c(p))$ such that $\lim_{n\rightarrow\infty}p_n=p_0\in \cF^c(p)$. If for some $p_n$ we have that $Dh^u_{x,p_n*}(\tilde v^u(x))\neq\pm Dh^s_{p,p_n*}(\tilde v^s(h^{s^{-1}}_{p,p_n}(p_n)))$ then the 2-dimensional unstable manifold of $x$, $W^u(x)$, intersects $\cF^c(p_n)$ in a $C^1$ curve locally transverse to the weak stable foliation inside $\cF^c(p_n)$ (which is the pushed forward by the stable holonomy of the weak stable foliation in $\cF^c(p)$). Since the 2-dimensional global stable manifold of $p$, $W^s(p)$, is dense inside the weak stable foliation of $\cF^c(p_n)$, we obtain a transverse homoclinic intersection from $x$ to $p$. 

Suppose that $W^u(x)\cap W^s(p)=\emptyset$. The above argument implies that
$$
Dh^u_{x,p_n*}(\tilde v^u)=\pm Dh^s_{p,p_n*}(\tilde v^s(h^{s^{-1}}_{p,p_n}(p_n)))\ \hbox{ for all }\ n\in\mathbb N.
$$
Since $f\in\mathcal U$, there exists a homoclinic point $q$ of $\cF^c(p)$ such that $\tilde g_q(p_0)\neq\pm\tilde v^s(p_0)\in T^1\cF^c(p)$. Let $q_0:=h^u_{p,q}(p_0)$, consider the strong unstable holonomy $h^u_{loc}:\cF^{cs}_{loc}(p_0)\rightarrow \cF^{cs}_{loc}(q_0)$ and let $q_n:=h^u_{loc}(p_n)\in \cF^{cs}_{loc}(q_0)$. Then $q_n\rightarrow q_0$. The lack of homoclinic relations between $x$ and $p$  implies that also
$$
Dh^u_{x,q_n*}(\tilde v^u)=\pm Dh^s_{p,q_n*}(\tilde v^s(h^{s^{-1}}_{p,q_n}(q_n)))\ \hbox{  for all }\ n\in\mathbb N.
$$
Since $h_{x,q_n}^u=h_{p_n,q_n}^u\circ h_{x,p_n}^u$ we obtain that
$$
Dh^s_{p,q_n*}(\tilde v^s(h^{s^{-1}}_{p,q_n}(q_n)))=\pm Dh^u_{p_n,q_n*}\circ Dh^s_{p,p_n*}(\tilde v^s(h^{s^{-1}}_{p,p_n}(p_n))).
$$
Using the continuity of $\tilde v^s$ and of $Dh^{s,u}$ we can pass to the limit and obtain that
$$
Dh^s_{p,q*}(\tilde v^s(h^{s^{-1}}_{p,q}(q_0)))=\pm Dh^u_{p,q*}(\tilde v^s(p_0)),
$$
or $\tilde g_q(p_0)=\pm\tilde v^s(p_0)$, which is a contradiction.

The proof of the intersection of the global stable manifold of $x$ with the global stable manifold of $p$ is similar. This concludes the proof.
\end{proof}

\section{Proof of Corollary~\ref{maincor.suppporthomoclinic}}

\begin{proof}
We have to show that for any $f\in\mathcal U^r$, the transverse homoclinic intersections of the invariant manifolds of the fixed hyperbolic point $p_f$ are dense in $\TT^4$. The proof uses the same ideas from the proof of Proposition \ref{prop:related}

Let $f\in \mathcal U^r$, and $p$ the hyperbolic fixed point of $f$ (for simplicity we will drop the index $f$ in the following arguments). Let $U$ be an open set in $\TT^4$. Since $W^u(p)\cap W^s(\cF^c(p))$ is dense in $\TT^4$, choose $x\in W^u(p)\cap W^s(\cF^c(p))$ such that $B(x,\delta)\subset U$ for some $\delta>0$. If $Dh^s_{p,x*}(\tilde v^s(h^s_{x,p}(x)))\notin T_xW^u(p)$, then clearly there is a transverse homoclinic intersection between $W^s(p)$ and $W^u(p)$ arbitrarily close to $x$. 

Suppose that $v:=Dh^s_{p,x*}(\tilde v^s(h^s_{x,p}(x)))\in T_xW^u(p)$. Then there exists a subsequence $n_k\rightarrow\infty$ and $(p_0,v_0)\in T^1\cF^c(p)$ such that $Df^{n_k}_*(x,v)\rightarrow (p_0,\tilde v^s(p_0))$. There exists a homoclinic point $q$ of $\cF^c(p)$ such that $\tilde g_q(p_0)\neq\pm\tilde v^s(p_0)\in T^1\cF^c(p)$. We consider again the strong unstable holonomy $h^u_{loc}:\cF^{cs}_{loc}(p_0)\rightarrow \cF^{cs}_{loc}(q_0)$ and let $q_k:=h^u_{loc}(f^{n_k}(x))\in \cF^{cs}_{loc}(q_0)$, where $q_0:=h^u_{p,q}(p_0)$. We have that
$$
Dh^s_{p,q*}(\tilde v^s(h^{s^{-1}}_{p,q}(q_0)))\neq\pm Dh^u_{p,q*}(\tilde v^s(p_0)),
$$
and by continuity for all $k$ large enough we have
$$
Dh^s_{p,q_k*}(\tilde v^s(h^{s^{-1}}_{p,q_k}(q_k)))\neq\pm Dh^u_{f^{n_k}(x),q_k*}(\tilde v^s(f^{n_k}(x))).
$$
Iterating by $f^{-n_k}$ and denoting $f^{-n_k}(q_k)=x_k\rightarrow x$ we obtain
$$
Dh^s_{p,x_k*}(\tilde v^s(h^{s^{-1}}_{p,x_k}(x_k)))\neq\pm Dh^u_{x,x_k*}(\tilde v^s(x))
$$
Since $Dh^u$ preserves $TW^u(p)$ we obtain that $Dh^s_{p,x_k*}(\tilde v^s(h^s_{x_k,p}(x_k)))\notin T_{x_k}W^u(p)$, with $x_k\in W^u(p)\cap W^s(\cF^c(p))$, and this implies again that arbitrarily close to $x_k$ (and thus close to $x$) there are transverse homoclinic intersection between $W^s(p)$ and $W^u(p)$. This finishes the proof.
\end{proof}

\section{Proof of Theorem~\ref{main.measured homoclinic }}

\begin{proof}
For simplicity, we may assume 
\begin{equation}\label{eq.boundpotention}
0\leq \inf \phi \leq \sup \phi \leq \log \lambda_B.
\end{equation}

First by the variation principle, there is a sequence of ergodic measures $\mu_n$ of $f$ such that
$$\limsup h_{\mu_n}=h_{top}(f)\geq \log \lambda_A+\log \lambda_B,$$ 
the last inequality comes from \eqref{eq.globalentropy}.

Again by the variation principle, pressure of the function $\phi$:
$$P_{top}(\phi) \geq \limsup (h_{\mu_n}+\int \phi d\mu_n)\geq \limsup h_{\mu_n}\geq \log \lambda_A+\log \lambda_B,$$
where the last inequality comes from the assumption that $\phi>0$ (\eqref{eq.boundpotention}).

Thus for any ergodic measure $\mu$ with pressure sufficient large, i.e.,
\begin{equation}\label{eq.lowbound}
h_\mu+\int \phi d\mu >P_{top}(\phi)+ (\int \phi d\mu -\log \lambda_B),
\end{equation}
we have 
$$h_\mu > P_{top}(\phi)-\log \lambda_B \geq \log \lambda_A.$$
As a consequence of Lemma~\ref{l.largeentropy}, $\mu$ is a hyperbolic measure with stable index $2$.
By Lemma~\ref{l.closing}, $\mu$ is homoclinically related to the atomic measure supported on a hyperbolic periodic point $O$.
Since $f\in \cV$, by Theorem~\ref{main.homoclinic intersection}, all the hyperbolic periodic orbits with stable index $2$ are homoclinically
related, as a consequence of Remark~\ref{rk.homoclinic measures}, all the hyperbolic ergodic measures with stable index $2$ are 
homoclinically related. In particular, all the ergodic measures satisfies \eqref{eq.lowbound} are homoclinically related.

Thus, all equilibrium states for the H\"older potential $\phi$ are homoclinically related, if they do exist. By Proposition~\ref{p.criterion},
there exists at most one equilibrium state. The proof is complete. 
\end{proof}

\bibliographystyle{plain}
\bibliography{shubclass}

\begin{thebibliography}{10}

\bibitem{A}
Carlos~F. \'{A}lvarez.
\newblock Hyperbolicity of maximal entorpy measures for certain maps isotopic
  to anosov.
\newblock {\em preprint arxiv:2011.06649 [math.DS]}, 2020.

\bibitem{Be1}
Snir Ben~Ovadia.
\newblock Symbolic dynamics for non-uniformly hyperbolic diffeomorphisms of
  compact smooth manifolds.
\newblock {\em J. Mod. Dyn.}, 13:43--113, 2018.

\bibitem{Be2}
Snir Ben~Ovadia.
\newblock The set of points with {M}arkovian symbolic dynamics for
  non-uniformly hyperbolic diffeomorphisms.
\newblock {\em Ergodic Theory Dynam. Systems}, 41(11):3244--3269, 2021.

\bibitem{Bo}
Rufus Bowen.
\newblock Entropy-expansive maps.
\newblock {\em Trans. Amer. Math. Soc.}, 164:323--331, 1972.

\bibitem{Br}
Aaron Brown.
\newblock Smoothness of stable holonomies inside center-stable manifolds and
  the $c^2$ hypothesis in pugh-shub and ledrappier-young theory.
\newblock {\em preprint arxiv:1608.05886 [math.DS]}, 2016.

\bibitem{Bu}
J\'{e}r\^{o}me Buzzi.
\newblock Intrinsic ergodicity of smooth interval maps.
\newblock {\em Israel J. Math.}, 100:125--161, 1997.

\bibitem{BCS}
J\'{e}r\^{o}me Buzzi, Sylvain Crovisier, and Omri Sarig.
\newblock Measures of maximal entropy for surface diffeomorphisms.
\newblock {\em Ann. of Math. (2)}, 195(2):421--508, 2022.

\bibitem{CP}
Maria Carvalho and Sebasti\'{a}n~A. P\'{e}rez.
\newblock Periodic points and measures for a class of skew-products.
\newblock {\em Israel J. Math.}, 245(1):455--500, 2021.

\bibitem{CT}
Vaughn Climenhaga and Daniel~J. Thompson.
\newblock Unique equilibrium states for flows and homeomorphisms with
  non-uniform structure.
\newblock {\em Adv. Math.}, 303:745--799, 2016.

\bibitem{FPS}
Todd Fisher, Rafael Potrie, and Mart\'{\i}n Sambarino.
\newblock Dynamical coherence of partially hyperbolic diffeomorphisms of tori
  isotopic to {A}nosov.
\newblock {\em Math. Z.}, 278(1-2):149--168, 2014.

\bibitem{Fra}
John Franks.
\newblock Anosov diffeomorphisms.
\newblock In {\em Global {A}nalysis ({P}roc. {S}ympos. {P}ure {M}ath., {V}ol.
  {XIV}, {B}erkeley, {C}alif., 1968)}, pages 61--93. Amer. Math. Soc.,
  Providence, R.I., 1970.

\bibitem{HPS}
M.~W. Hirsch, C.~C. Pugh, and M.~Shub.
\newblock {\em Invariant manifolds}.
\newblock Lecture Notes in Mathematics, Vol. 583. Springer-Verlag, Berlin-New
  York, 1977.

\bibitem{K}
A.~Katok.
\newblock Lyapunov exponents, entropy and periodic orbits for diffeomorphisms.
\newblock {\em Inst. Hautes \'{E}tudes Sci. Publ. Math.}, (51):137--173, 1980.

\bibitem{LY}
F.~Ledrappier and L.-S. Young.
\newblock The metric entropy of diffeomorphisms. {II}. {R}elations between
  entropy, exponents and dimension.
\newblock {\em Ann. of Math. (2)}, 122(3):540--574, 1985.

\bibitem{LS}
Fran\c{c}ois Ledrappier and Jean-Marie Strelcyn.
\newblock A proof of the estimation from below in {P}esin's entropy formula.
\newblock {\em Ergodic Theory Dynam. Systems}, 2(2):203--219 (1983), 1982.

\bibitem{LW}
Fran\c{c}ois Ledrappier and Peter Walters.
\newblock A relativised variational principle for continuous transformations.
\newblock {\em J. London Math. Soc. (2)}, 16(3):568--576, 1977.

\bibitem{LVY}
Gang Liao, Marcelo Viana, and Jiagang Yang.
\newblock The entropy conjecture for diffeomorphisms away from tangencies.
\newblock {\em J. Eur. Math. Soc. (JEMS)}, 15(6):2043--2060, 2013.

\bibitem{M}
Ricardo Ma\~{n}\'{e}.
\newblock Contributions to the stability conjecture.
\newblock {\em Topology}, 17(4):383--396, 1978.

\bibitem{Man}
Anthony Manning.
\newblock There are no new {A}nosov diffeomorphisms on tori.
\newblock {\em Amer. J. Math.}, 96:422--429, 1974.

\bibitem{NY}
Sheldon~E. Newhouse and Lai-Sang Young.
\newblock Dynamics of certain skew products.
\newblock In {\em Geometric dynamics ({R}io de {J}aneiro, 1981)}, volume 1007
  of {\em Lecture Notes in Math.}, pages 611--629. Springer, Berlin, 1983.

\bibitem{PYY}
Maria~Jose Pacifico, Fan Yang, and Jiagang Yang.
\newblock Existence and uniqueness of equilibrium states for systems with
  specification at a fixed scale: an improved {C}limenhaga-{T}hompson
  criterion.
\newblock {\em Nonlinearity}, 35(12):5963--5992, 2022.

\bibitem{PSW}
Charles Pugh, Michael Shub, and Amie Wilkinson.
\newblock Partial differentiability of invariant splittings.
\newblock {\em J. Statist. Phys.}, 114(3-4):891--921, 2004.

\bibitem{PS}
Enrique~R. Pujals and Mart\'{\i}n Sambarino.
\newblock Homoclinic tangencies and hyperbolicity for surface diffeomorphisms.
\newblock {\em Ann. of Math. (2)}, 151(3):961--1023, 2000.

\bibitem{Sh}
M.ichael Shub.
\newblock Topologically transitive diffeomorphisms of t4.
\newblock In David Chillingworth, editor, {\em Proceedings of the Symposium on
  Differential Equations and Dynamical Systems}, pages 39--40, Berlin,
  Heidelberg, 1971. Springer Berlin Heidelberg.

\bibitem{Sm}
Stephen Smale.
\newblock Diffeomorphisms with many periodic points.
\newblock In {\em Differential and {C}ombinatorial {T}opology ({A} {S}ymposium
  in {H}onor of {M}arston {M}orse)}, pages 63--80. Princeton Univ. Press,
  Princeton, N.J., 1965.

\bibitem{TY}
Ali Tahzibi and Jiagang Yang.
\newblock Invariance principle and rigidity of high entropy measures.
\newblock {\em Trans. Amer. Math. Soc.}, 371(2):1231--1251, 2019.

\bibitem{W}
Lan Wen.
\newblock Homoclinic tangencies and dominated splittings.
\newblock {\em Nonlinearity}, 15(5):1445--1469, 2002.

\bibitem{Ya}
Jiagang Yang.
\newblock Entropy along expanding foliations.
\newblock {\em Adv. Math.}, 389:Paper No. 107893, 39, 2021.

\bibitem{Y}
Y.~Yomdin.
\newblock Volume growth and entropy.
\newblock {\em Israel J. Math.}, 57(3):285--300, 1987.

\end{thebibliography}

\end{document}